\documentclass[reqno,10pt]{amsart}

\usepackage{amsthm,amsmath,amsfonts,amssymb}
\usepackage{latexsym}
\usepackage{mathrsfs}
\newtheorem{theorem}{Theorem}

\newtheorem{proposition}[theorem]{Proposition}
\newtheorem{remark}[theorem]{Remark}
\newtheorem{lemma}[theorem]{Lemma}

\newcommand{\pvor}{\qep}
\newcommand{\vph}{\varphi}
\newcommand{\A}{\mathbb{A}}
\newcommand{\Aep}{\mathbb{A}_\ep}
\renewcommand{\H}{H}
\newcommand{\Hep}{H}

\newcommand{\Hi}{H_{\infty}}

\newcommand{\Pep}{\mathbb{P}_\ep}
\newcommand{\Piep}{\Pi_\ep}

\newcommand{\n}{\mathbf{n}}

\renewcommand{\d}{\mathrm{d}}

\newcommand{\real}{\mathbb{R}}
\newcommand{\R}{\mathbb{R}}
\newcommand{\loc}{\scriptstyle{loc}}

\newcommand\al{\alpha}
\newcommand\ep{\varepsilon}
\newcommand\Deltaep{\Delta_{\Pi_\ep}}
\newcommand\Deltar{\Delta_{\R^2}}
\newcommand\uep{u_\ep}
\newcommand\uept{\widetilde{u}_\ep}
\newcommand\vep{v_\ep}

\newcommand\wep{w_\ep}
\newcommand\wept{\widetilde{w}_\ep}
\newcommand\qep{q_\ep}
\newcommand\qept{\widetilde{q}_\ep}

\newcommand\etamu{\eta_\mu}

\newcommand\phimu{\varphi_\mu}
\newcommand\piep{\Pi_\ep}

\renewcommand{\S}{\operatorname{\mathbf{S}_\ep}}
\newcommand{\T}{\operatorname{\mathbf{T}_\ep}}
\newcommand{\K}{\operatorname{\mathbf{K}_\ep}}
\newcommand{\G}{\operatorname{\mathbf{G}_\ep}}
\newcommand{\supp}{\operatorname{supp}}
\newcommand{\ga}{\gamma}
\newcommand{\qb}{\overline{q}}
\newcommand{\ub}{\overline{u}}
\newcommand{\vb}{\overline{v}}

\DeclareMathOperator{\dv}{div}
\DeclareMathOperator{\dive}{div}
\DeclareMathOperator{\curl}{curl}

\usepackage{color}

\title[Incompressible $\al$--Euler outside vanishing disk]{The incompressible $\al$--Euler equations in the exterior of a vanishing disk}

\author[Busuioc \and Iftimie \and Lopes Filho \and Nussenzveig Lopes]{A.V. Busuioc \and D. Iftimie \and M.C. Lopes Filho \and H.J. Nussenzveig Lopes}

\begin{document}

\begin{abstract}
In this article we consider the  $\al$--Euler equations in the exterior of a small fixed disk of radius $\ep$. We assume that the initial potential vorticity is compactly supported and independent of $\ep$, and that the circulation of the unfiltered velocity on the boundary of the disk does not depend on $\ep$. We prove that the solution of this problem converges, as $\ep\to0$, to the solution of a modified $\al$--Euler equation in the full plane where an additional Dirac located at the center of the disk is imposed in the potential vorticity.
\end{abstract}

\maketitle

\section{Introduction}

In this work we study the initial-boundary-value problem for the two-dimensional incompressible  $\al$--Euler equations, $\alpha > 0$ fixed, in the exterior of the small disk $D(0;\ep)= \{x \in \real^2 \; | \; |x| \leq \ep \}$.

Let $\Pi_\ep \equiv \{|x|>\ep\}$. The system we are interested in is given by:
\begin{equation}\label{EulerAlphaVelVare}
\left\{
\begin{array}{ll}
\partial_t \vep + \uep \cdot \nabla \vep + \sum_{j=1}^2 (\vep)_j \nabla (\uep)_j = -\nabla p_\ep, & \text{ in } (0,\infty ) \times \Pi_\ep,\\
\dv \uep = 0,  & \text{ in }  [0,\infty ) \times \Pi_\ep,\\
\vep= \uep - \alpha \Delta \uep,  & \text{ in }  [0,\infty ) \times \Pi_\ep ,\\
\uep = 0,  & \text{ on }  [0,\infty ) \times \{|x|=\ep\},\\
\lim\limits_{|x|\to\infty}\uep(t,x) = 0,  & \text{ for all  }  t\geq0,\\
\uep(0,\cdot)=u_{\ep,0},  & \text{ at }  \{t=0\}  \times \Pi_\ep.
\end{array}
\right.
\end{equation}
Above $\uep$ is called the {\it filtered} velocity while $\vep$ is the {\it unfiltered} velocity.

The $\al$--Euler equations arise in several ways: as the inviscid case of the second-grade fluid model, see \cite{dunn_thermodynamics_1974}, averaging the transporting velocity in the Euler equations at scale $\sqrt{\alpha}$, as the equation for geodesics in the group of volume-preserving diffeomorphisms with a natural metric, see \cite{marsden_geometry_2000}, or as a variant of the vortex blob method, see \cite{oliver_vortex_2001}.

The $\al$-Euler equations are a natural desingularization of the inviscid flow equations, obtained by averaging momentum transport at small scales, away from solid boundaries. In domains with boundary a boundary condition must be imposed; a natural choice is to impose the no-slip condition $u=0$ on the filtered velocity, something which makes the $\al$-Euler equations into a rough analog of the standard initial-boundary value problem for the Navier-Stokes equations. Recent progress has been obtained in understanding the flow-boundary interaction for this $\al$-model, focusing mainly on the vanishing $\al$ limit, see \cite{LNZT15,busuioc_weak_2017,BILN20}.

The present work is part of this program, seeking to identify the limiting behavior of the flow in the exterior of a small obstacle for fixed $\al$. This is inspired by work of Iftimie {\em et al.} in two space dimensions, where this limit was identified both for the Euler and Navier-Stokes equations, see \cite{iftimie_two_2003,ILN2006}. The limit is sharply different in the inviscid and viscous cases. For inviscid flow, the small obstacle leads to a modified Euler system, whereas for viscous flow, it is the initial condition that must be adjusted. Given that the $\al$-Euler model is a regularized inviscid system using the standard viscous boundary condition, it is natural to wonder whether the present limit follows the inviscid pattern, the viscous one, or something else altogether. Our main result is that the limit follows the inviscid pattern.

The proof also follows the structure of the corresponding result for the Euler equations, with additional complications coming from potential theory. Dealing with these complications makes up the bulk of the present work.

Let us note that, taking the two-dimensional curl of \eqref{EulerAlphaVelVare}, which
corresponds to applying the differential operator $\nabla^\perp \cdot$ to the system, gives rise to the {\it potential
vorticity equation}:

\begin{equation} \label{pvorVareWOBdryCond}
\left\{
\begin{array}{ll}
\partial_{t}\pvor + \uep \cdot\nabla\pvor = 0, & \text{ in }  (0,\infty ) \times \Pi_\ep,\\
\dv \uep=0, & \text{ in } [0,\infty ) \times \Pi_\ep,\\
\curl (1-\alpha\Delta)\uep = \curl \vep = \pvor, & \text{ in } [0,\infty ) \times \Pi_\ep\\
\uep = 0,  & \text{ on }  [0,\infty ) \times \{|x|=\ep\},\\
\lim\limits_{|x|\to\infty}\uep(t,x) = 0,  & \text{ for all  }  t\geq0,\\
\pvor(0,\cdot)=q_{\ep,0},  & \text{ at } \{0\}  \times \Pi_\ep.
\end{array}
\right.
\end{equation}

The scalar quantity $\pvor$ is the {\em potential vorticity}. 

We will work with the vorticity equation rather than the velocity equation. This requires a modified Biot-Savart law expressing the velocity $u_\ep$ in terms of the vorticity $q_\ep$. Since the domain we are considering is not simply-connected, we require additional information to determine velocity from vorticity. We will impose a given circulation of the unfiltered velocity $\vep$ on the boundary; we denote this circulation by $\gamma$. This is a conserved quantity for the evolution, as noted in \cite[Lemma 2.3]{busuioc_weak_2017}.
We will see, in Section \ref{sect2}, that the velocity $\uep$ is uniquely determined in terms of  $\qep$ and of $\ga$. We will show that the modified Biot-Savart $\uep=\T(\qep)$ law, which gives the velocity $u_\ep$ in terms of the potential vorticity $q_\ep$ and of the circulation $\gamma$, is:
\begin{equation}\label{BSeq1}
\uep=\T(\qep)\equiv(1+\al\Aep)^{-1}[\K(\qep)+(\ga+m)\H]
\end{equation}
where $\Aep$ is the Stokes operator, $m=\int_{\piep}\qep$ (another conserved quantity), $\H$ is the following harmonic vector field
\begin{equation*}
H=\frac{x^\perp}{2\pi|x|^2}
\end{equation*}
 and $\K(\qep)$ is the classical Biot-Savart law in $\piep$:
\begin{equation}\label{classicalKep}
\K(\qep)(x)=\int_{\piep}\nabla_x^\perp G_\ep(x,y)\qep(y)\,dy.
\end{equation}
Above $G_\ep$ denotes the Green's function for the Laplacian on $\Pi_\ep$ with zero boundary conditions.

The purpose of this article is to prove the following result.

\begin{theorem} \label{mainthm}
Assume that the initial potential vorticity $q_{\ep,0}=q_0\in L^1\cap L^\infty$ is compactly supported outside the origin and independent of $\ep$. Assume, in addition, that the circulation of the unfiltered velocity $v_\ep$ on the boundary of $\piep$ is a constant $\gamma$ independent of $\ep$, so that the velocity $u_\ep$ can be expressed from the potential vorticity with the Biot-Savart law \eqref{BSeq1}: $u_\ep=\T(q_\ep)$. Then
\begin{enumerate}
\item There exists a unique global solution $\qep$ of \eqref{pvorVareWOBdryCond} with $u_\ep=\T(q_\ep)$, such that  $q^\ep\in L^\infty(\R_+;L^1(\piep)\cap L^\infty(\piep))$.
\item Let $\qept$ be the extension of $\qep$ to $\R^2$ which coincides with $\qep$ in $\piep$ and vanishes in $|x|\leq \ep$. Then we have that  $\qept \rightharpoonup q$ weak-$\ast$ in $L^\infty(\R_+;L^1(\R^2)\cap L^\infty(\R^2))$, as $\ep\to 0$, and $q$ is a global solution of the  following system of PDE in the full plane:
\begin{equation} \label{limitsystem} 
\left\{
\begin{array}{ll}
\partial_{t}q + u \cdot\nabla q = 0, & \text{ in }  (0,\infty ) \times \R^2,\\
\dv u=0, & \text{ in } [0,\infty ) \times\R^2,\\
\curl (1-\alpha\Delta)u = q+\gamma\delta, & \text{ in } [0,\infty ) \times \R^2,\\
q(0,\cdot)=q_{0},  & \text{ at } \{0\}  \times \Pi_\ep,
\end{array}
\right.
\end{equation}
and $u \in L^\infty_{loc}(\R_+;L^p(\R^2))$ for all $p>2$.
\item The limit system \eqref{limitsystem} has at most one global solution $q\in L^\infty(\R_+;L^1(\R^2)\cap L^\infty(\R^2))$.
\end{enumerate}
\end{theorem}

We note that the limit system \eqref{limitsystem} above is not the $\al$--Euler equations in $\R^2$. Indeed, one could consider the $\al$--Euler equations in $\R^2$
with initial potential vorticity $\overline{q}_0=q_0+\ga\delta$. Such an initial data is a bounded measure so it produces a unique global solution $\qb=q+\gamma\delta_{z(t)}$ with $z(0)=0$, see \cite{oliver_vortex_2001}. It is easy to see that the PDE for the regular part $q$ is the same as \eqref{limitsystem} except that one has to write $\delta_{z(t)}$ instead of $\delta$. In contrast with \eqref{limitsystem}, for the $\al$--Euler equations the position $z(t)$ of the discrete part is no longer constant and must evolve along the trajectories of the velocity associated to the regular part $q$. So, although these two equations are very similar, they are not the same.

The plan of this paper is the following. In Section \ref{sect2} we introduce notation, we deduce the modified Biot-Savart law \eqref{BSeq1} and we establish some $\ep$-dependent estimates. Global existence and uniqueness for system \eqref{pvorVareWOBdryCond}, for fixed $\ep$, is shown in Section \ref{sect3}. Next, we prove estimates uniform in $\ep$ in Section \ref{sect4}. The convergence result is the subject of Section \ref{sect5}. Finally, the uniqueness of solutions of \eqref{limitsystem} is shown in Section \ref{sect6}.

\section{Notation, modified Biot-Savart law and  preliminary estimates}
\label{sect2}

We begin by introducing basic notation. Given a Banach space $X$ of vector fields on $\Pi_\ep$ we denote by  $X_{\sigma}$ the subspace of $X$ consisting of divergence-free vector fields in $X$ which are tangent to the boundary.

Recall that $\Pi_\ep$ is the exterior of the disk of radius $\ep$. We use the subscript $\ep$ to denote the dependence of solutions of the  $\al$--Euler equations on the domain, as in $\uep$, $\vep$, $\pvor$. We will assume that $\ep$ is small.

We choose a smooth radial function $\eta\in C^\infty(\R^2;[0,1])$ such that $\eta(x)=0$ for \label{eta} all $|x|\leq 1$ and $\eta(x)=1$ for all $|x|\geq 2$. We define
\begin{equation}\label{defH}
H=\frac{x^\perp}{2\pi|x|^2}
\end{equation}
and
 \begin{equation*}
\Hi\equiv\eta(x)H.
\end{equation*}
The radial symmetry of $\eta$ implies that $\Hi$ is divergence free.

We denote by $\Pep$ the {\em Leray projector} in $\piep$, i.e. the orthogonal projection from $L^2(\Pi_\ep)$ to $L^2_{\sigma}(\Pi_\ep)$. It is well known that the Leray projector can be extended by density to a continuous projection from $L^p(\Pi_\ep)$ to $L^p_{\sigma}(\Pi_\ep)$, for all $1< p < \infty$. The Stokes operator on $\piep$ with homogeneous Dirichlet boundary conditions is denoted by $\A_\ep=-\Pep\Delta$. When we apply these operators, or their inverses, to functions defined on the whole $\R^2$ we mean to apply them to the restrictions of these functions to $\piep$.

We will use the following result about the inverse of the operator $1+\al\Aep$.
\begin{proposition}\label{Stokes}
The operator $(1+\al\Aep)^{-1}$ is bounded from $L^p_\sigma(\Piep)$ to $W^{2,p}(\Piep)\cap W^{1,p}_0(\Piep)\cap L^p_\sigma(\Piep)$ for any $1<p<\infty$ and from $L^\infty_\sigma(\Piep)$ to $W^{1,\infty}(\Piep)$. In addition, there exists a universal constant $C_0$ such that
\begin{equation}
\label{linfb}
\|(1+\al\Aep)^{-1}f\|_{L^\infty(\Piep)}\leq C_0 \|f\|_{L^\infty(\Piep)}\quad\text{for all }f\in L^\infty_\sigma(\Piep).
\end{equation}
\end{proposition}
\begin{proof}
The statement regarding $L^p_\sigma$ can be found in \cite[Corollary 5]{hieber_stokes_2018} and the boundedness from $L^\infty_\sigma$ to $W^{1,\infty}$ is proved in \cite{abe_stokes_2015}.

To prove the bound \eqref{linfb} we observe first that it was already proved in \cite{abe_stokes_2015} that \eqref{linfb} holds true with a constant $C_0=C_0(\ep)$ which may depend on $\ep$ but not on $\al$. An immediate scaling argument shows that it is also independent of $\ep$.  Indeed, changing variables $x=\ep x'$ the operator $1+\al\Aep$ becomes $1+\frac\al{\ep^2}\A_1$. So we can apply the result of \cite{abe_stokes_2015} to the operator $(1+\frac\al{\ep^2}\A_1)^{-1}$ on $L^\infty_\sigma(\Pi_1)$ and find a universal constant $C_0$ as an upper bound for its norm as a bounded operator on $L^\infty_\sigma(\Pi_1)$. By the rescaling performed, this universal constant $C_0$ is also an upper bound for the norm of $(1+\al\Aep)^{-1}$ in  $L^\infty_\sigma(\Piep)$.
\end{proof}

The integral of the potential vorticity is a constant of motion. Indeed,
\[
\frac{\d}{\d t}\int_{\Pi_\ep} \pvor(t,x)\,\d x = \int_{\Pi_\ep} \partial_t \pvor \,\d x = -\int_{\Pi_\ep} \dv (\uep  \pvor ) \,\d x =0,
\]
since $\uep$ vanishes at $|x|=\ep$. We denote
\begin{equation*}
m\equiv \int_{\Pi_\ep} q_\ep=\int_{\Pi_\ep} q_{0}.
\end{equation*}




We will later need some detailed information on the operator $(1 - \alpha \Delta)^{-1}$ in the full plane. This is a convolution operator against a kernel, denoted $\mathcal{G}_\alpha$, which is a re-scaled  {\em Bessel potential}. The classical Bessel potential from Harmonic Analysis, $\mathcal{J}_2$, is the kernel for $(1-\Delta)^{-1}$, and we have
\begin{equation*}
 \mathcal{G}_\alpha(x) = \frac{1}{\alpha} \mathcal{J}_2\left(\frac{x}{\sqrt{\alpha}}\right).
\end{equation*}

We will make use of the following properties of $\mathcal{G}_\alpha$, deduced from those satisfied by $\mathcal{J}_2$; the first three can be found in Chapter V.3.1 of \cite{stein_singular_1970}, and, for the fourth property, see \cite{Watson}, page 80, relation (14).

\begin{enumerate}
\item[(P1)] $\mathcal{G}_{\alpha}$ is radially symmetric; i.e. $\mathcal{G}_\alpha(x)=g_\alpha(|x|)$ for some
$g_\alpha = g_\alpha(r)$;
\item[(P2)] $g_{\alpha}$ is positive and
\begin{equation} \label{galphaposandint1}
\int_{\real^2} \mathcal{G}_\alpha =  2\pi \int_0^\infty sg_\alpha (s) \, ds = 1;
\end{equation}
\item[(P3)] $\mathcal{G}_{\alpha}$ decays  exponentially at infinity, i.e. for any  $M>0$, there exist positive constants
$c_{1}$ and $c_{2}$ such that
\begin{equation*}
g_{\alpha}(|x|) \leq  c_{1} e^{-c_{2}|x|}, \mbox{   whenever } \ |x|>M;
\end{equation*}
\item[(P4)] $\mathcal{G}_\alpha$ has a logarithmic singularity at $0$, i.e., there exists $c_3  \in \real$, such that
\begin{equation} \label{Galphalogat0}
g_{\alpha}(|x|) = c_3 \log|x| + \mathcal{O}(1), \text{ as } |x| \to 0.
\end{equation}
Additionally, $g_\alpha$ is bounded for $|x|\geq 1/2$.
\end{enumerate}

Let us introduce now the following kernel
\begin{equation} \label{Kaldef}
K^\al=\mathcal{G}_\alpha \ast \H.
\end{equation}
We will need, later, the following estimates for  $K^\al$.
\begin{lemma}\label{Kal}
There exists a constant  $C>0$, which depends only on $\al$, such that:
\begin{enumerate}
\item If $|x|<1/2$ we have:
$$|K^\al(x)|\leq C|x|\bigl|\log|x|\bigr|\quad\text{and}\quad |\nabla K^\al(x)|\leq C\bigl|\log|x|\bigr|.$$
\item For all $x\in\R^2$, we have that $|K^\al(x)|\leq C/(1+|x|)$.
\item We have that
\begin{equation}\label{newest}
|\partial_2K^\al_1(x)+\partial_1K^\al_2(x)|\leq C\quad\text{and}\quad |\partial_1K^\al_1(x)-\partial_2K^\al_2(x)|\leq C
\end{equation}
for all $x\in\R^2$.
\end{enumerate}
\end{lemma}
\begin{proof}
There is a simple way to express $K^\al(x)$ in terms of $g_\al$, see for instance \cite[page 5476]{ambrose_confinement_2018-1}, namely:
\begin{equation}\label{Kalform}
K^\al(x)=\frac{x^\perp}{|x|^2}\int_0^{|x|}sg_\al(s)\,ds.
\end{equation}
Let $x$ be such that $|x|<1/2$. It follows from property (P4), \eqref{Galphalogat0}, that
\begin{equation*}
|K^\al(x)|\leq\frac{C}{|x|}\int_0^{|x|}s|\log s|\,ds\leq C|x|\bigl|\log|x|\bigr|,
\end{equation*}
for some constant $C>c_3$.

Differentiating \eqref{Kalform} and estimating the result as above gives the gradient estimate and proves part a).

The bound $C/(1+|x|)$ in part b) follows from the estimates given in \cite[page 5476]{ambrose_confinement_2018-1}.

To prove part c) we compute $\partial_2K^\al_1+\partial_1K^\al_2$ using \eqref{Kalform}. We find, after some calculations, that:
\begin{equation*}
\partial_2K^\al_1(x)+\partial_1K^\al_2(x)=\frac{x_1^2-x_2^2}{|x|^2}\left[ g_\al(|x|)-\frac2{|x|^2}\int_0^{|x|}s g_\al(s)\,ds \right].
\end{equation*}
We use again (P4), \eqref{Galphalogat0}, to obtain that
\begin{multline*}
g_\al(|x|)-\frac2{|x|^2}\int_0^{|x|}s g_\al(s)\,ds \\ = c_3\log(|x|)-c_3\frac2{|x|^2}\int_0^{|x|}s \log(s)\,ds+O(1)\\
= O(1)\text{ as } |x|\to 0.
\end{multline*}
It follows that $\partial_2K^\al_1+\partial_1K^\al_2$ is bounded for small $x$. The global bound is a consequence of property (P3) and the boundedness of $g_\alpha$ for $|x|\geq 1/2$.  The corresponding estimate for $\partial_1K^\al_1-\partial_2K^\al_2$ follows in the same manner.
\end{proof}

We will now deduce the modified Biot-Savart law, which expresses the velocity $\uep$ in terms of the potential vorticity $\qep$ and the circulation of $\vep$ around the boundary of $\piep$. Our point of departure is the following elliptic system, which relates potential vorticity to the unfiltered velocity:
\[
\left\{
\begin{array}{ll}
\dv \vep=0, & \text{ in } [0,\infty ) \times \Pi_\ep,\\
\curl \vep = \pvor, & \text{ in } [0,\infty ) \times \Pi_\ep.\\
\end{array}
\right.
\]

We recall that, above, $\vep = (1-\alpha\Delta)\uep $, and $\uep$ satisfies the boundary condition
\[
\uep = 0,  \text{ on }  [0,\infty ) \times \{|x|=\ep\}.
\]

 Then, since $\Pep \vep$ and $\vep$ differ by a gradient, it follows easily that
\[
\left\{
\begin{array}{ll}
\dv (\Pep\vep)=0, & \text{ in } [0,\infty ) \times \Pi_\ep,\\
\curl (\Pep \vep) = \pvor, & \text{ in } [0,\infty ) \times \Pi_\ep,\\
\Pep \vep \cdot \widehat{\n} = 0 & \text{ on } [0,\infty ) \times \{|x|=\ep\}.
\end{array}
\right.
\]

The system above was studied in detail in \cite{iftimie_two_2003}. It was shown, see \cite[page 358]{iftimie_two_2003}, that there exists $\beta_\ep = \beta_\ep (t)\in \real$ for which
\[\Pep \vep = \K(\pvor) + \beta_\ep (t) \Hep,
\]
where the operator $\K (\pvor)$ is the classical Biot-Savart law defined in \eqref{classicalKep}, and $\H$  is the generator of the harmonic vector fields in $\Pi_\ep$ with unit circulation defined in \eqref{defH}.

We now argue that  $\beta_\ep=\ga+m $, where $\ga$ is the circulation of the unfiltered velocity $\vep$ on the boundary of $\piep$ and $m$ is the mass of vorticity: $m=\int_{\piep}\qep$.  Following the proof of \cite[Lemma 3.1]{iftimie_two_2003} we have that
\[\beta_\ep = \int_{\{| x |=\ep\}} \Pep \vep \cdot \d s + \int_{\Pi_\ep} \pvor \, \d x.\]
First observe that  $\vep = \Pep \vep + \nabla Q$, for some smooth function $Q$, and that
\[\int_{\{| x |=\ep\}} \nabla Q \cdot \d s = 0,
\]
since $\{|x |=\ep\}$ is a closed curve. Therefore
\[\int_{\{| x |=\ep\}} \Pep \vep \cdot \d s = \int_{\{|x|=\ep\}} \vep \cdot \d s \equiv \gamma,
\]
which is a conserved quantity, see \cite[Lemma 2.3]{busuioc_weak_2017}.


In summary, we have shown that
\begin{equation*}
 \Pep \vep = \K(\pvor) + (\ga+m) \H.
\end{equation*}
We regard the expression on the right-hand-side above as an operator acting on $\qep$, which we denote by $\S$:
\begin{equation*}
\qep \mapsto \S(\qep)\equiv \K(\qep)+(\gamma+m)\H.
\end{equation*}


Then
\begin{equation*}
\S(\qep)=\Pep \vep=\Pep(\uep-\al\Delta \uep)=\uep+\al\Aep \uep.
\end{equation*}

Inverting the operator $1+\al\Aep$ allows to deduce the following modified Biot-Savart law
\begin{equation}\label{mBS}
\uep=\T(\qep)=(1+\al\Aep)^{-1}\S(\qep)=(1+\al\Aep)^{-1}[\K(\qep)+(\gamma+m)\H].
\end{equation}

In the following Proposition we collect some estimates related to this modified Biot-Savart law.
\begin{proposition}\label{h1boundT}
For all $q\in L^1(\Piep)\cap L^\infty(\Piep)$ such that $\int_{\Piep}q=m$ we have that $\T(q)\in W^{1,\infty}(\Piep)$ and
\begin{equation}\label{bounduinfty}
\|\T(q)\|_{W^{1,\infty}(\Piep)}\leq C_1(\|q\|_{L^1(\Piep)\cap L^\infty(\Piep)} +|\gamma|),
\end{equation}
where the constant $C_1>0$ depends only on $\al$ and $\ep$. If we assume, in addition, that $\supp q\subset \{|x|\leq R\}$ for some finite $R$, then we also have $\T(q)-(\gamma+m)\H\in W^{2,p}(\Piep)$ for all $1<p<\infty$ and
\begin{equation*}
\|\T(q)-(\gamma+m)\H\|_{W^{2,p}(\Piep)}\leq C_2 (\|q\|_{L^1(\Piep)\cap L^\infty(\Piep)} +|\gamma|),
\end{equation*}
where the constant $C_2>0$ depends only on $\al,\ep,R$ and $p$.
\end{proposition}
\begin{proof}
It follows from \cite[Theorem 4.1]{iftimie_two_2003} that the quantity $\K(q)+m\H$ is  bounded in $\Piep$, and the $L^\infty$-norm may be estimated by $C\|q\|_{L^1(\piep)\cap L^\infty(\piep)}$, where $C$ is a universal constant. Clearly
\begin{equation*}
\|\S(q)\|_{L^\infty}\leq\|\K(q)+m\H\|_{L^\infty}+\|\gamma\H\|_{L^\infty}\leq C\|q\|_{L^1(\Piep)\cap L^\infty(\Piep)} +\frac{|\ga|}{2\pi\ep}.
\end{equation*}
Since $\T(q)= (1+\al\Aep)^{-1}\S(q)$, relation \eqref{bounduinfty} follows from Proposition \ref{Stokes}.

\medskip

Assume now that $\supp q\subset \{|x|\leq R\}$. Then we know, from \cite[relation (2.8)]{iftimie_two_2003}, that $\K(q)$ is bounded by $C(\ep,R)/|x|^2$, so it belongs to $L^p(\Piep)$ for all $p>1$. Therefore $\S(q)-(\gamma+m)\H=\K(q)\in L^p(\Piep)$ for all $p>1$. Then Proposition \ref{Stokes} implies that
\begin{equation*}
\T(q)-(\gamma+m)(1+\al\Aep)^{-1}\H
=(1+\al\Aep)^{-1}(\S(q)-(\gamma+m)\H)\in W^{2,p}(\Piep)
\end{equation*}
for all $1<p<\infty$. Finally, we observe that $(1+\al\Aep)^{-1}(\Hi-\al\Delta \Hi)=\Hi$ and we write
\begin{align*}
\T(q)-(\gamma+m)\H
&=\T(q)-(\gamma+m)(1+\al\Aep)^{-1}\H+(\gamma+m)[(1+\al\Aep)^{-1}\H-\H]\\
&=\T(q)-(\gamma+m)(1+\al\Aep)^{-1}\H\\
&\hskip 1,5cm +(\gamma+m)[\Hi-\H+(1+\al\Aep)^{-1}(\H-\Hi+\al\Delta \Hi)]\\
&\in W^{2,p}(\Piep)
\end{align*}
for all $1<p<\infty$. We used above that $\Hi-\H\in W^{2,p}(\Piep)$, that $\H-\Hi+\al\Delta \Hi\in L^p_\sigma(\Piep)$ and Proposition \ref{Stokes}. This completes the proof.
\end{proof}

\section{Existence of the solution for fixed $\ep$}
\label{sect3}

In this section we prove part a) of Theorem \ref{mainthm}. A similar result, requiring more regularity and with a different proof can be found in \cite{YouZang21}.

We want to solve the following problem
\begin{equation}\label{eqtosolve}
\left\{
\begin{split}
\partial_t q+\T(q)\cdot\nabla q&=0, \qquad t>0,|x|>\ep\\
q\bigl|_{t=0}&=q_0, \qquad |x|>\ep
\end{split}
\right.
\end{equation}
where $q_0\in L^1(\Piep)\cap L^\infty(\Piep)$.

We construct a recursive sequence of approximate solutions in the following manner. We set $q^0(t,x)=q_0(x)$ and $u^0(t,x)=\T(q^0)(x)$. We define $q^1=q^1(t,x)$, $q^1\in L^\infty((0,+\infty);L^1\cap L^\infty (\Piep))$, to be the unique weak solution of the following transport equation
\begin{equation*}
\left\{
\begin{split}
\partial_t q^1+u^0\cdot\nabla q^1&=0, \qquad t>0,|x|>\ep\\
q^1\bigl|_{t=0}&=q_0, \qquad |x|>\ep.
\end{split}
\right.
\end{equation*}
The global existence of $q^1$ follows from Picard's theorem since, by Proposition \ref{h1boundT}, we have that $u^0$ is a Lipschitz vector field so its flow map is globally well-defined. Uniqueness can be established by energy estimates.

Next, given $q^n\in L^\infty((0,+\infty);L^1\cap L^\infty (\Piep))$, we define $q^{n+1}\in L^\infty((0,+\infty);L^1\cap L^\infty (\Piep))$ recursively as the unique global weak solution of the transport equation
\begin{equation}\label{eqqn}
\left\{
\begin{split}
\partial_t q^{n+1}+u^n\cdot\nabla q^{n+1}&=0, \qquad t>0,|x|>\ep\\
q^{n+1}\bigl|_{t=0}&=q_0, \qquad |x|>\ep,
\end{split}
\right.
\end{equation}
where $u^n=\T(q^n)$.
We will  pass to the limit in the above problem and find a solution of \eqref{eqtosolve}.

We first observe that, since $q^n$ satisfies a transport equation by a divergence-free vector field, the integral of $q^n$ is conserved:
\begin{equation*}
\int_{\Piep}q^n(t,x)\,dx=\int_{\Piep}q^n(0,x)\,dx=\int_{\Piep}q_0(x)\,dx=m.
\end{equation*}
We have that the circulation of $\K(q^n)$ on the boundary is equal to the quantity $\left( -\int_{\Piep}q^n(t,x)\,dx \right)$ (see  the proof of \cite[Lemma 3.1]{iftimie_two_2003}) and we recall that $H$ has unit circulation on the boundary. With this we infer that the circulation of $\S(q^n)$ on the boundary is $\gamma$.

Using, again, that $q^n$ satisfies a transport equation, we have that $q^n$ is bounded in $L^\infty(\R_+;L^1\cap L^\infty)$ uniformly in $n$. Proposition \ref{h1boundT} implies that $u^n$ is bounded in $L^\infty(\R_+;W^{1,\infty}(\Piep))$ uniformly in $n$. Indeed:
\begin{equation}\label{Lipb}
\begin{aligned}
\|u^n(t,\cdot)\|_{W^{1,\infty}(\Piep)}
&\leq C_1(\al,\ep)(\|q^n\|_{L^1(\Piep)\cap L^\infty(\Piep)} +|\gamma|)\\
&= C_1(\al,\ep)(\|q_0\|_{L^1(\Piep)\cap L^\infty(\Piep)} +|\gamma|)\\
&\equiv C_3.
\end{aligned}
\end{equation}
Since $q^{n+1}$ is transported by the flow of $u^n$, we deduce that $\supp q^{n+1}(t,\cdot)\subset\{|x|\leq R_0+C_3t\}$ where  $\supp q_0\subset\{|x|\leq R_0\}$. We use the second part of Proposition \ref{h1boundT} to obtain that, for all $1<p<\infty$, the quantity $u^n-(\gamma+m)\H$ is bounded in $L^\infty_{loc}([0,\infty);W^{2,p}(\Piep))$, uniformly in $n$. We will  prove that it actually converges in $H^2$.

Let us introduce the notation $v^{n+1}=(1-\al\Delta)u^{n+1}$ and
$$\vb^n=\Pep v^n=\S(q^n)=(1+\al\Aep)u^n.$$
We also define $\xi^n=\Delta_{\Piep}^{-1}q^n=\G(q^n)$, where $\G$ is the operator with kernel $G_\ep$, the Green function on $\Piep$. Let us observe that
\begin{equation*}
\vb^{n+1}-\vb^n=\S(q^{n+1})-\S(q^n)=\K(q^{n+1}-q^n)=\nabla^\perp(\xi^{n+1}-\xi^n).
\end{equation*}
In addition, since $q^n$ and $q^{n+1}$ are compactly supported and $\nabla^\perp G_\ep$, the kernel of $\K$, decays like $1/|x|^2$ at infinity (see \cite[relation (2.8)]{iftimie_two_2003}) we have that $\vb^{n+1}-\vb^n$ also decays like $1/|x|^2$ at infinity. In particular, it belongs to $L^2(\Piep)$.

We subtract \label{page1} the equation for $q^n$ from the equation for $q^{n+1}$ (see relation \eqref{eqqn}), we multiply it by $\xi^{n+1}-\xi^n$ and integrate in space (which is possible because $q^n$ and $q^{n+1}$ are compactly supported). We then obtain
\begin{align*}
\int_{\Piep}\partial_t(q^{n+1}-q^n)(\xi^{n+1}-\xi^n)+\int_{\Piep}(u^n\cdot\nabla q^{n+1}- u^{n-1}\cdot\nabla q^{n})(\xi^{n+1}-\xi^n)=0.
\end{align*}
The integrals above should be interpreted as duality pairings between $W^{-1,p}$ and $W^{1,p^\prime}$, where $p^\prime = p/(p-1)$ and $p \in (1,\infty)$.

Clearly
\begin{align*}
\int_{\Piep}\partial_t(q^{n+1}-q^n)(\xi^{n+1}-\xi^n)
&=\int_{\Piep}\Delta[\partial_t(\xi^{n+1}-\xi^n)](\xi^{n+1}-\xi^n)\\
&=-\int_{\Piep}[\partial_t\nabla(\xi^{n+1}-\xi^n)]\nabla(\xi^{n+1}-\xi^n)\\
&=-\frac12\frac{\d}{\d t}\|\nabla(\xi^{n+1}-\xi^n)\|_{L^2(\Piep)}^2\\
&=-\frac12\frac{\d}{\d t}\|\vb^{n+1}-\vb^n\|_{L^2(\Piep)}^2.
\end{align*}

Therefore
\begin{align*}
\frac12\frac{\d}{\d t}\|\vb^{n+1}&-\vb^n\|_{L^2(\Piep)}^2
= \int_{\Piep}(u^n\cdot\nabla q^{n+1}- u^{n-1}\cdot\nabla q^{n})(\xi^{n+1}-\xi^n)\\
&= \int_{\Piep}u^n\cdot\nabla (q^{n+1}- q^{n})(\xi^{n+1}-\xi^n) + \int_{\Piep}(u^n- u^{n-1})\cdot\nabla q^{n}(\xi^{n+1}-\xi^n)\\
&= \int_{\Piep}u^n\cdot(\vb^{n+1}- \vb^{n})^{\perp}(q^{n+1}-q^n) + \int_{\Piep}(u^n- u^{n-1})\cdot (\vb^{n+1}-\vb^n)^{\perp}q^n\\
&\equiv I_1+I_2,
\end{align*}
where we integrated by parts and used that $\nabla(\xi^{n+1}-\xi^n)=-(\vb^{n+1}- \vb^{n})^{\perp}$. We will now estimate these two terms.

Observe first that for a divergence free vector field $h$ we have the identity
\begin{equation*}
h^{\perp}\curl h=
\begin{pmatrix}
\partial_2\\\partial_1
\end{pmatrix}(h_1h_2)
+\begin{pmatrix}
-\partial_1\\\partial_2
\end{pmatrix}\frac{h_2^2-h_1^2}2.
\end{equation*}
Recalling  that $q^{n+1}-q^n=\curl(\vb^{n+1}- \vb^{n})$ and integrating by parts we find:
\begin{align*}
|I_1|
&= \left| \int_{\Piep}u^n\cdot(\vb^{n+1}- \vb^{n})^{\perp}(q^{n+1}-q^n) \right|\\
&= \left|\int_{\Piep}(\partial_2 u^n_1+\partial_1 u^n_2)(\vb_1^{n+1}- \vb_1^{n})(\vb_2^{n+1}- \vb_2^{n}) \right.\\
&\left.\hskip3cm + \frac12\int_{\Piep}(-\partial_1 u^n_1+\partial_2 u^n_2)\bigl[(\vb_2^{n+1}- \vb_2^{n})^2-(\vb_1^{n+1}- \vb_1^{n})^2\bigr]\right|\\
&\leq C\|\nabla u^n\|_{L^\infty}\|\vb^{n+1}- \vb^{n}\|_{L^2}^2\\
&\leq C\|\vb^{n+1}- \vb^{n}\|_{L^2}^2,
\end{align*}
where we used \eqref{Lipb}. We estimate now $I_2$:
\begin{align*}
|I_2|&\leq \|u^n-u^{n-1}\|_{L^2}\|\vb^{n+1}-\vb^n\|_{L^2}\|q^n\|_{L^\infty}\\
&\leq C\|(1+\al\Aep)^{-1}(\vb^n-\vb^{n-1})\|_{L^2}\|\vb^{n+1}-\vb^n\|_{L^2}\\
&\leq C\|\vb^n-\vb^{n-1}\|_{L^2}\|\vb^{n+1}-\vb^n\|_{L^2}\\
&\leq C\|\vb^n-\vb^{n-1}\|_{L^2}^2+C\|\vb^{n+1}-\vb^n\|_{L^2}^2.
\end{align*}

We conclude that\label{page2}
\begin{equation*}
\frac{\d}{\d t}\|\vb^{n+1}-\vb^n\|_{L^2(\Piep)}^2
\leq C_4\|\vb^n-\vb^{n-1}\|_{L^2}^2+C_4\|\vb^{n+1}-\vb^n\|_{L^2}^2,
\end{equation*}
for some constant $C_4$ independent of $n$ and $t$. Given that $\vb^{n+1}-\vb^n$ vanishes at the initial time, the Gronwall lemma implies the following bound:
\begin{equation*}
\sup_{[0,T]}\|\vb^{n+1}-\vb^n\|_{L^2(\Piep)}\leq \sup_{[0,T]}\|\vb^{n}-\vb^{n-1}\|_{L^2(\Piep)}\sqrt{e^{C_4T}-1}.
\end{equation*}

We choose the time $T_0$ such that $\sqrt{e^{C_4T_0}-1}=\frac12$. Then, using the estimate above recursively we find, by induction, that
\begin{equation*}
\sup_{[0,T_0]}\|\vb^{n+1}-\vb^n\|_{L^2(\Piep)}\leq 2^{-n}\sup_{[0,T_0]}\|\vb^{1}-\vb^{0}\|_{L^2(\Piep)}.
\end{equation*}
It then follows, from Proposition \ref{Stokes}, that
\begin{equation*}
\sup_{[0,T_0]}\|u^{n+1}-u^n\|_{H^2(\Piep)}\leq C 2^{-n}.
\end{equation*}
This means that the sequence $u^n-(\ga+m)\H$ is a Cauchy sequence in the space  $C^0([0,T_0];H^2(\Piep))$ and, therefore, that it is convergent. Hence there exists some $u$ such that $u-(\ga+m)\H\in C^0([0,T_0];H^2(\Piep))$ and
\begin{equation}\label{weakun}
u^n-u\to0\quad\text{strongly in }C^0([0,T_0];H^2(\Piep)).
\end{equation}
Denoting $q=\curl(u-\al\Delta u)$ we further obtain that
\begin{equation*}
q^n-q\to0\quad\text{strongly in }C^0([0,T_0];H^{-1}(\Piep)).
\end{equation*}
In particular $q(0,x)=q_0(x)$.

The sequence $q^n$ being bounded in $L^\infty(\R_+;L^1(\Piep)\cap L^\infty(\Piep))$ implies that $q\in L^\infty(\R_+;L^1(\Piep)\cap L^\infty(\Piep))$. Moreover, there exists a subsequence $q^{n_k}$ such that
\begin{equation}\label{weakqn}
q^{n_k}\rightharpoonup q\quad\text{weak}*\text{ in }L^\infty(\R_+;L^1\cap L^\infty).
\end{equation}

To complete the proof of the existence part of Theorem \ref{mainthm}, it remains to prove that $u=\T(q)$. We know that $u-(\gamma+m)\H\in L^\infty([0,T_0];H^2(\Piep))$. From Proposition \ref{h1boundT} we also know that $\T(q)-(\gamma+m)\H\in L^\infty([0,T_0];H^2(\Piep))$ so we must have that $u-\T(q)\in L^\infty([0,T_0];H^2(\Piep))$. Moreover, $u-\T(q)$ is divergence free and vanishes at the boundary.

Let $\varphi\in C^\infty_{c,\sigma}([0,T_0]\times\Piep)$. There exists  $\psi\in C^\infty_{c}([0,T_0]\times\overline{\Pi}_\ep)$ such that $\varphi=\nabla^\perp\psi$. Then $\nabla \psi=0$ in a neighborhood of the boundary of $\Piep$, so $\psi$ must be constant in the same neighborhood. We denote by $\overline\psi(t)$ this constant, so that $\psi(t,x) \equiv \overline\psi(t)$ in a neighborhood of  the boundary of $\Piep$.

By self-adjointness of the Stokes operator we have that
\begin{equation}\label{star}
\begin{split}
\int_0^{T_0}\int_{\Piep}u^{n_k}\cdot (1+&\al\Aep)\varphi
=\int_0^{T_0}\int_{\Piep}(1+\al\Aep)u^{n_k}\cdot \varphi\\
&=\int_0^{T_0}\int_{\Piep}\S(q^{n_k})\cdot \varphi\\
&=\int_0^{T_0}\int_{\Piep}\S(q^{n_k})\cdot \nabla^\perp\psi\\
&=-\int_0^{T_0}\int_{\Piep}\curl \S(q^{n_k})\ \psi-\int_0^{T_0}\int_{\partial\Piep}\frac{x^\perp}{|x|}\cdot\S(q^{n_k})\ \psi\\
&=-\int_0^{T_0}\int_{\Piep}q^{n_k} \psi-\gamma\int_0^{T_0} \overline\psi(t)\,dt,
\end{split}
\end{equation}
where we used that the circulation of $\S(q^{n_k})$ on the boundary is $\gamma$.

Since $\Delta \varphi$ is divergence free and tangent to the boundary (actually compactly supported), we have that $\Pep \Delta \varphi=\Delta \varphi$. So $(1+\al\Aep)\varphi=\varphi-\al\Delta\varphi\in C^\infty_{c,\sigma}([0,T_0]\times\Piep)$. The convergence properties expressed in relations \eqref{weakun} and \eqref{weakqn} allow to pass to the limit $k\to\infty$ in \eqref{star} to obtain that
\begin{equation*}
\int_0^{T_0}\int_{\Piep}u\cdot (1+\al\Aep)\varphi
=-\int_0^{T_0}\int_{\Piep}q \psi-\gamma\int_0^{T_0} \overline\psi(t)\,dt.
\end{equation*}
But the same calculations as in \eqref{star} show that
\begin{equation*}
\int_0^{T_0}\int_{\Piep}\T(q)\cdot (1+\al\Aep)\varphi
=-\int_0^{T_0}\int_{\Piep}q \psi-\gamma\int_0^{T_0} \overline\psi(t)\,dt.
\end{equation*}
We deduce that
\begin{equation*}
\int_0^{T_0}\int_{\Piep}u\cdot (1+\al\Aep)\varphi
=\int_0^{T_0}\int_{\Piep}\T(q)\cdot (1+\al\Aep)\varphi
\end{equation*}
which means that
\begin{equation*}
\int_0^{T_0}\int_{\Piep}(1+\al\Aep)u\cdot \varphi
=\int_0^{T_0}\int_{\Piep}(1+\al\Aep)\T(q)\cdot \varphi.
\end{equation*}
This implies that $(1+\al\Aep)(u-\T(q))=0$, so necessarily $u=\T(q)$.

We proved that there exists a solution of \eqref{pvorVareWOBdryCond} on the time interval $[0,T_0]$. Repeating the argument starting from $T_0$ we can extend this solution up to time $2T_0$. Continuing like this we  construct a global solution. Its uniqueness is classical. It can be proved by estimating the $H^2$ norm of the difference of two solutions, with exactly the same argument as in the estimate of $\frac{\d}{\d t}\|\vb^{n+1}-\vb^n\|_{L^2(\Piep)}^2$ so we omit it. Part a) of Theorem \ref{mainthm} is now proved.

\section{$H^1$ estimates for $\uep$}
\label{sect4}

The aim of this section is to derive estimates for $\uep$ which are independent of $\ep$.

We start by noting that, since $\qep$ satisfies a transport equation by a divergence-free vector field, we have
\begin{align*}
\|\qep\|_{L^\infty(\R_+;L^1(\piep)\cap L^\infty(\piep))}
&\leq \|q_{\ep,0}\|_{L^\infty(\R_+;L^1(\piep)\cap L^\infty(\piep))}\\
&= \|q_{0}\|_{L^\infty(\R_+;L^1(\piep)\cap L^\infty(\piep))},
\end{align*}
which is bounded independently of $\ep$.

Let us now make the following observation. If $f$ is a scalar radial function, decaying sufficiently fast at infinity, then $(1+\alpha \A_\ep )^{-1}(x^\perp f)$ is of the form $x^\perp g$, where $g$ is a scalar radial function. Indeed, let $h=(1 - \alpha \Deltaep)^{-1}(x^\perp f)$. Then, for any special orthogonal matrix $M$ (rotation matrix), the vector field $P(x)=x^\perp f$ is invariant under the transformation $P(x)\to M^t P(Mx)$. The rotational invariance of the Laplacian allows to write the following sequence of computations:
\begin{align*}
(1-\al\Delta)[M^t h(Mx)]
&=M^t(1-\al\Delta)[h(Mx)]\\
&=M^t[(1-\al\Delta)h](Mx)\\
&=M^tP(Mx)\\
&=P(x)\\
&=(1-\al\Delta)h(x).
\end{align*}
We infer that $ M^t h(Mx)=h(x)$, so $h(Mx)=Mh(x)$ for every special orthogonal matrix $M$. This means that $h$ must be of the form $h=x^\perp g_1+xg_2$ with $g_1,g_2$ scalar radial functions. But one can check that $(1 - \alpha \Delta)(xg_2)$ is proportional to $x$ and $(1 - \alpha \Delta)(x^\perp g_1)$ is proportional to $x^\perp$. Since the sum must be $x^\perp f$ we infer that $xg_2=0$, so $h=x^\perp g_1$. Recalling that $g_1$ is radial, we observe that $\dive h=\dive(x^\perp g_1)=0$. We conclude that $h-\al\Delta h=x^\perp f$, that $h$ vanishes at $|x|=\ep$ and at infinity and that $h$ is divergence free. Therefore $h=(1+\alpha \A_\ep )^{-1}(x^\perp f)$. We proved, in this paragraph, that if $f$ is a scalar radial function then $(1+\alpha \A_\ep )^{-1}(x^\perp f)=(1-\alpha \Deltaep )^{-1}(x^\perp f)$ is of the form $x^\perp$ times a scalar radial function.

We will now proceed to obtain the required {\em a priori} estimates for $\uep$.
We start with the following lemma.
\begin{lemma} \label{WvareAsymp}
We have that $\K(q_\ep) + m \Hep$ is bounded in $L^\infty(\R_+\times\piep)$ and  $\K(q_\ep) + m (\Hep-\Hi)$ is bounded in $L^\infty_{loc}([0,\infty);L^2(\piep))$, independently of $\ep$.
\end{lemma}
\begin{proof}
We recall the result in \cite[Theorem 4.1]{iftimie_two_2003}, which reads, in our notation:
\[\|\K(q_\ep) + m \Hep \|_{L^\infty(\Pi_\ep)}\leq C\|q_\ep\|_{L^\infty}^{1/2}\|q_\ep\|_{L^1}^{1/2}\leq C\|q_0\|_{L^\infty}^{1/2}\|q_0\|_{L^1}^{1/2},\]
where $C>0$ is independent of $\ep$. This shows the first bound.

Next, we need to obtain bounds on the support of $\qep$ which are independent of $\ep$. We write, thanks to the modified Biot-Savart law \eqref{mBS},
\begin{equation*}
\uep=(1+\al\Aep)^{-1}[\K(q_\ep) + m \Hep]+\ga(1+\al\Aep)^{-1}\H\equiv \uep^1+\uep^2.
\end{equation*}

We bound $\uep^1$ by using relation \eqref{linfb}:
\begin{align*}
\|\uep^1\|_{L^\infty(\piep)}
&=\|(1+\al\Aep)^{-1}[\K(q_\ep) + m \Hep]\|_{L^\infty(\piep)}\\
&\leq C_0\|\K(q_\ep) + m \Hep\|_{L^\infty(\piep)}\\
&\leq C_0C\|q_0\|_{L^\infty}^{1/2}\|q_0\|_{L^1}^{1/2}\\
&\equiv M.
\end{align*}

We observe now that since $H$ is of the form $x^\perp$ times a radial function, then $\uep^2=(1+\al\Aep)^{-1}\H$ is of the same form. So the trajectories of $\uep^2$ are circles centered in the origin. Recall that $\qep$ is transported by the velocity field $\uep=\uep^1+\uep^2$. If we want to estimate how far from the origin the support of $\qep$ can go, we can ignore the term $\uep^2$. Since we bounded $\uep^1$ by $M$, we infer that
\begin{equation}\label{suppq}
\supp q_\ep(t,\cdot)\subset \{|x|\leq R_0+Mt\}
\end{equation}
where $R_0$ is such that $\supp q_0\subset \{|x|\leq R_0\}$.

Next we remark that $ \K(q_\ep) + m(\Hep-\Hi)=(\K(q_\ep) + m\Hep)-m\Hi$ is a sum of uniformly bounded vector fields, hence bounded in $L^2_{\loc} (\Pi_\ep)$ independently of $\ep$. Furthermore, using the expressions for the kernel $K_\ep(x,y)\equiv\nabla_x^\perp G_\ep(x,y)$ of the Biot-Savart law $\K$ given in  \cite[(2.5)]{iftimie_two_2003} with the conformal mapping given by $T(x)=x/\ep$ implies the following formula
\begin{equation*}
K_\ep(x,y)=\frac{(x-y)^\perp}{2\pi|x-y|^2}
-\frac{(x-\ep^2y/|y|^2)^\perp}{2\pi|x-\ep^2y/|y|^2|^2}\cdot
\end{equation*}
Using the relation $\bigl|\frac{a}{|a|^2}-\frac{b}{|b|^2}\bigr|=\frac{|a-b|}{|a||b|}$ and observing that $|\ep^2y/|y|^2|\leq|y|$ we get for $|x|>2|y|$ the following upper bound for the kernel $K_\ep$:
\begin{equation*}
|K_\ep(x,y)|=\frac{|y-\ep^2y/|y|^2|}{2\pi|x-y||x-\ep^2y/|y|^2|}
\leq \frac{|y|+|\ep^2y/|y|^2|}{2\pi(|x|-|y|)(|x|-|\ep^2y/|y|^2|)}\leq \frac{4|y|}{\pi|x|^2}.
\end{equation*}
Recalling the bound on the support of $\qep$ given in \eqref{suppq}, we can now estimate for any $|x|>2(R_0+Mt)$:
\begin{align*}
|\K(\qep)(x)|
&\leq\int_{\piep}|K_\ep(x,y)|\, |\qep(y)|\,dy\\
&\leq\int_{\piep}\frac{4|y|}{\pi|x|^2}\, |\qep(y)|\,dy\\
&\leq \frac{4(R_0+Mt)\|\qep\|_{L^1(\piep)}}{\pi|x|^2}.
\end{align*}
Observing that $\H-\Hi$ vanishes for $|x|>2$  establishes the second bound in $L^2(\Pi_\ep)$ as desired.
\end{proof}



Recalling the modified Biot-Savart law \eqref{mBS}
we can rewrite $\uep$ in the following way:
\begin{equation}\label{decuep}
\uep=(1+\alpha \A_\ep )^{-1} [\K(q_\ep) + m (\Hep-\Hi)]
+\ga(1+\alpha \A_\ep )^{-1}\Hep+m (1+\alpha \A_\ep )^{-1}\Hi.
\end{equation}

We will now proceed to derive $H^1$ estimates for each of these terms.

We start with the first one.
\begin{lemma} \label{H1part}
We have that  $(1+\alpha \A_\ep )^{-1}(\K(q_\ep) + m (\Hep-\Hi))$ is bounded in the space $L^\infty_{loc}([0,\infty);H^1(\piep))$ independently of $\ep$.
\end{lemma}

\begin{proof}
Let us denote $b_\ep=\K(q_\ep) + m (\Hep-\Hi)$ and $a_\ep=(1+\alpha \A_\ep )^{-1}b_\ep$. Then $(1+\alpha \A_\ep)a_\ep=b_\ep$. We do a classical energy estimate in which we multiply this relation by $a_\ep$. We get
\begin{equation*}
\int_{\piep}a_\ep\cdot b_\ep=\int_{\piep}a_\ep\cdot(1+\alpha \A_\ep)a_\ep=\int_{\piep}(|a_\ep|^2-\al a_\ep\cdot\Delta a_\ep)=\int_{\piep}(|a_\ep|^2+\al |\nabla a_\ep|^2).
\end{equation*}
But
\begin{equation*}
\int_{\piep}a_\ep\cdot b_\ep\leq \|a_\ep\|_{L^2(\piep)}^{\frac12}\|b_\ep\|_{L^2(\piep)}^{\frac12}
\leq \Bigl(\int_{\piep}(|a_\ep|^2+\al |\nabla a_\ep|^2)\Bigr)^{\frac12}\|b_\ep\|_{L^2(\piep)}^{\frac12}
\end{equation*}
so
\begin{equation*}
\int_{\piep}(|a_\ep|^2+\al |\nabla a_\ep|^2)\leq \int_{\piep}|b_\ep|^2.
\end{equation*}
From the previous lemma we know that $b_\ep$ is bounded in $L^2$ independently of $\ep$, so the conclusion follows.
\end{proof}

Next we will discuss the third term in \eqref{decuep}. To this end we perform an energy estimate for the term  $(1+\alpha \A_\ep )^{-1} \Hi$.
\begin{lemma}\label{Hipart}
There exists a constant $C=C(\al)$ which depends only on $\al$ such that
\begin{equation*}
\|(1+\alpha \A_\ep )^{-1} \Hi-\Hi\|_{H^1(\Pi_\ep)}\leq C(\al).
\end{equation*}
\end{lemma}
\begin{proof}
Since $\Hi$ is of the form $x^\perp$ times a radial function, we have that $w_1\equiv (1+\alpha \A_\ep )^{-1} \Hi=(1-\alpha \Deltaep )^{-1}\Hi$ solves the following boundary value problem:
\begin{equation*}
\left\{\begin{aligned}
w_1-\al\Delta w_1=&\Hi\text{ for }|x|>\ep\\
w_1\bigl|_{|x|=\ep}=&0.
\end{aligned}
\right.
\end{equation*}
Then $w_2\equiv w_1-\Hi$ satisfies
\begin{equation*}
\left\{\begin{aligned}
w_2-\al\Delta w_2=&\al\Delta\Hi\text{ for }|x|>\ep\\
w_2\bigl|_{|x|=\ep}=&0.
\end{aligned}
\right.
\end{equation*}
The same $L^2$ estimate as in the proof of Lemma \ref{H1part} shows that
\begin{equation*}
\|w_2\|_{L^2(\Pi_\ep)}^2+\al\|\nabla w_2\|_{L^2(\Pi_\ep)}^2 \leq\al^2\|\Delta\Hi\|_{L^2(\Pi_\ep)}^2.
\end{equation*}
Since $\Hi=\eta\Hep$ and $\Hep$ is harmonic we have that
$$\Delta\Hi=\Delta\eta\, \Hep+2\nabla\eta\cdot\nabla\Hep$$
is a $C^\infty$ function compactly supported in the annulus $\{1\leq|x|\leq2\}$ which does not depend on $\ep$. The conclusion follows.
\end{proof}

To complete the estimate of $\uep$, it remains only to bound the term $(1 + \alpha \A_\ep)^{-1} \Hep$. The strategy in deriving this estimate is to compare $(1 + \alpha \A_\ep)^{-1}$, in $\Pi_\ep$, to $(1-\alpha \Deltar)^{-1}$ in all of $\real^2$. We recall that the vector field $K^\al$ was defined in relation \eqref{Kaldef}.
\begin{proposition}\label{Hpart}
There exists a constant  $C >0$, which depends only on $\al$, such that
\begin{equation*}
\|(1+\alpha \A_\ep )^{-1}  \Hep-  K^\al \|_{H^1(\Pi_\ep)}\leq C\ep|\log\ep|.
\end{equation*}
\end{proposition}
\begin{proof}
Since $\Hep$ is of the form $x^\perp$ times a radial function, we have that $w_3\equiv (1+\alpha \A_\ep )^{-1} \Hep=(1-\alpha \Deltaep )^{-1}\Hep$ is of the same form and solves the following boundary value problem:
\begin{equation*}
\left\{
\begin{array}{ll}
(1 - \alpha \Delta) w_3  = \H, & \text{ in } \Pi_\ep \\
\dv w_3 = 0, & \text{ in } \Pi_\ep\\
w_3 = 0  , & \text{ on } \partial \Pi_\ep.
\end{array}
\right.
\end{equation*}
Since $\H$ is divergence free in the whole of $\R^2$, we have that $K^\al=\mathcal{G}_\alpha \ast \H $ is also divergence free everywhere. Then $w_4\equiv w_3- K^\al $ satisfies the following boundary-value problem:
\begin{equation} \label{Nsystem}
\left\{
\begin{array}{ll}
(1 - \alpha \Delta) w_4  = 0, & \text{ in } \Pi_\ep \\
\dv w_4 = 0, & \text{ in } \Pi_\ep\\
w_4 = - K^\al  , & \text{ on } \partial \Pi_\ep.
\end{array}
\right.
\end{equation}

Set
\begin{equation}\label{starFF}
F \equiv \curl w_4.
\end{equation}

Taking the $\curl$ of the first equation in \eqref{Nsystem} we obtain
\begin{equation}
\label{star2}
F - \alpha \Delta F = 0\quad\text{ in }\Pi_\ep.
\end{equation}

Note that, because $\dv w_4 = 0$, we have
\begin{equation}\label{curlLaplacian}
  \nabla^\perp F = \Delta w_4.
\end{equation}

Evaluating the first equation in \eqref{Nsystem} on $\partial \Pi_\ep$, using \eqref{curlLaplacian} and taking the inner product of the result with the unit vector $x^\perp / \ep$, which is tangent to $\partial \Pi_\ep$, yields
\[w_4 \cdot \frac{x^\perp}{\ep} - \alpha \nabla^\perp F \cdot \frac{x^\perp}{\ep} = 0 \quad \text{ on } |x|=\ep.\]

Set $\widehat{\n} \equiv x / \ep$, the {\em interior} unit normal to $\partial\Pi_\ep$. Then, since $\nabla^\perp F \cdot x^\perp / \ep = \nabla F \cdot \widehat{\n} = \partial F/\partial{\widehat{\n}}$, we have derived a Neumann boundary condition for $F$:
\begin{equation}\label{NeumannCond}
\frac{\partial F }{\partial {\widehat{\n}}} = -\frac{1}{\alpha} ( K^\al  )\cdot \widehat{\n}^\perp\qquad\text{at }|x|=\ep.
\end{equation}

Next we use the circular symmetry of $\H $ together with the radial symmetry of $\mathcal{G}_\alpha$ to deduce that $ K^\al  $ is a vector of the form $x^\perp$ times a radial function, see \eqref{Kalform}. Hence the right-hand-side of \eqref{NeumannCond} is a real number which depends only on $\alpha$, which is fixed, and on $\ep$. We denote it by $A_\ep$:
\begin{equation*}
  A_\ep \equiv -\frac{1}{\alpha} K^\al \big|_{|x|=\ep}\cdot \frac{x^\perp}{\ep}.
\end{equation*}

With this notation we note, in particular, that
\begin{equation} \label{NjBdry}
  w_4\big|_{\partial \Pi_\ep} = \frac{\alpha}{\ep} A_\ep  x^\perp.
\end{equation}
We also observe at this point that, thanks to Lemma \ref{Kal}, we have that
\begin{equation}\label{sizeaep}
A_\ep=O(\ep|\log\ep|)\quad\text{as }\ep\to0.
\end{equation}

We have, thus far, deduced a Neumann boundary-value problem for $F$:
\begin{equation*}
\left\{
\begin{array}{ll}
(1 - \alpha \Delta) F  = 0, & \text{ in } \Pi_\ep \\
\frac{\partial F }{\partial {\widehat{\n}}} = A_\ep , & \text{ on } \partial \Pi_\ep.
\end{array}
\right.
\end{equation*}

We already observed that $w_3$ and $K^\al$ are of the form $x^\perp$ times a radial function, so $w_4=w_3-K^\al$ is also of this form. Therefore  $F = \curl w_4$ is radially symmetric in $\Pi_\ep$. Hence the restriction of $F$ to the boundary $\partial \Pi_\ep$ depends only on $\alpha$, fixed, and on $\ep$. We denote this constant by $B_\ep$:
\begin{equation}\label{bvare}
 B_\ep \equiv F\big|_{\partial\Pi_\ep}.
\end{equation}

Our next step is to extend $F$ to all of $\real^2$, find an equation satisfied by this extension and solve it. To this end consider the continuous extension of $F$ given by
\begin{equation}
\label{star3}
\overline{F}  \equiv \left\{
                            \begin{array}{l}
                                F, \text{ if } |x|>\ep\\
                                B_\ep , \text{ if } |x|\leq \ep.
                            \end{array}
                       \right.
\end{equation}

Let us compute $\Delta \overline{F} $ in the sense of distributions on $\R^2$. Fix $\vph \in C^\infty_c(\real^2)$. Then
\begin{align*}
\langle \Delta \overline{F} , \vph \rangle
&= \langle \overline{F} , \Delta \vph \rangle\\
&= \int_{\{|x|>\ep\}} F \Delta \vph + \int_{\{|x|\leq \ep\}} B_\ep \Delta \vph  \\
&= - \int_{\{|x|>\ep\}}\nabla F \nabla \vph + \int_{\{|x|=\ep\}}F\nabla\vph\cdot (-\widehat{\n}) +
B_\ep \int_{\{|x|= \ep\}} \nabla \vph \cdot \widehat{n}   \\
&=  \int_{\{|x|>\ep\}}\vph \Delta F   - \int_{\{|x|=\ep\}} \vph \nabla F  \cdot (-\widehat{\n}) \\
&=  \int_{\{|x|>\ep\}}\vph \Delta F   +A_\ep  \int_{\{|x|=\ep\}} \vph.
\end{align*}

Recalling \eqref{star2} and \eqref{star3} we infer that
\begin{equation*}
(1-\alpha\Delta)\overline{F}  = B_\ep  \chi_{_{\{|x|\leq \ep\}}} - \alpha A_\ep  \delta_{\{|x|=\ep\}}\quad\text{in } \mathscr{D}'(\R^2).
\end{equation*}

We can now invert the operator $1-\alpha\Deltar$ to find a formula for $\overline{F} $:
\begin{equation*}
\overline{F}  = \mathcal{G}_\alpha \ast \left(B_\ep  \chi_{_{\{|x|\leq \ep\}}} - \alpha A_\ep  \delta_{\{|x|=\ep\}}\right).
\end{equation*}

Evaluating the above expression at 0 we have, by definition of $\overline{F} $,
\[ B_\ep  = \overline{F} (0)= B_\ep  \int_{\{|x|\leq \ep\}} \mathcal{G}_\alpha - \alpha A_\ep  \int_{\{|x|=\ep\}}
\mathcal{G}_\alpha \,ds.\]

It follows, see property (P2) for $\mathcal{G}_\alpha$, that
\begin{equation*} 
  B_\ep  = -\alpha A_\ep  \frac{\int_{\{|x|=\ep\}}
\mathcal{G}_\alpha \,ds}{\int_{\{|x|> \ep\}} \mathcal{G}_\alpha }.
\end{equation*}

Using the properties of the kernel $\mathcal{G}_\alpha$ and the estimate of the size of $A_\ep$ given in \eqref{sizeaep} we can estimate the size of $B_\ep$. More precisely, relation \eqref{galphaposandint1} implies  that
\begin{equation*}
\int_{\{|x|> \ep\}} \mathcal{G}_\alpha \to \int_{\R^2} \mathcal{G}_\alpha=1 \quad\text{as }\ep\to0,
\end{equation*}
and, thanks to \eqref{Galphalogat0}, we have that
\begin{equation*}
\int_{\{|x|=\ep\}}\mathcal{G}_\alpha \,ds=O(\ep|\log\ep|).
\end{equation*}
Combining these two bounds with \eqref{sizeaep} implies that
\begin{equation}
\label{sizebep}
B_\ep=O(\ep^2|\log\ep|^2) \quad\text{as }\ep\to0.
\end{equation}

\medskip

Knowing $B_\ep$, we can compute the $H^1$ norm of $w_4$. We multiply  the first equation of \eqref{Nsystem} by $w_4$ and integrate on $\piep$. We obtain
\begin{equation*}
0=\int_{\Pi_\ep}(w_4-\al\Delta w_4)\cdot w_4
=\int_{\Pi_\ep}|w_4|^2+\al\int_{\Pi_\ep}|\nabla w_4|^2+\al\int_{\partial \Pi_\ep}\frac{\partial w_4 }{\partial {\widehat{\n}}}\cdot w_4.
\end{equation*}
Given that $w_4$ is of the form $x^\perp$ times by a radial function, we can show, through an easy calculation, the following identity:
\begin{equation*}
x\cdot\nabla w_4=x^\perp\curl w_4-w_4.
\end{equation*}
Recalling \eqref{NjBdry}, \eqref{bvare} and \eqref{starFF} we observe that
\begin{align*}
\int_{\partial \Pi_\ep}\frac{\partial w_4 }{\partial {\widehat{\n}}}\cdot w_4
&=\frac1\ep \int_{\partial \Pi_\ep}x\cdot\nabla w_4\cdot w_4\\
&=\frac1\ep \int_{\partial \Pi_\ep}(x^\perp\curl w_4-w_4)\cdot w_4\\
&=\frac1\ep \int_{\partial \Pi_\ep}(B_\ep-\frac\al\ep A_\ep)x^\perp\cdot x^\perp\frac\al\ep A_\ep\\
&=2\pi\al\ep A_\ep(B_\ep-\frac\al\ep A_\ep).
\end{align*}

Finally, we conclude that
\begin{equation*}
\int_{\Pi_\ep}|w_4|^2+\al\int_{\Pi_\ep}|\nabla w_4|^2
=-\al\int_{\partial \Pi_\ep}\frac{\partial w_4 }{\partial {\widehat{\n}}}\cdot w_4=2\pi\al^2\ep A_\ep(\frac\al\ep A_\ep-B_\ep).
\end{equation*}

Recalling that $w_4=(1+\alpha \A_\ep )^{-1}  \Hep-  K^\al$ and using \eqref{sizeaep} and \eqref{sizebep} completes the proof.
\end{proof}

We summarize the results of this section in the result below.

\begin{theorem} \label{uvareW1plocests}
We have that $\uep-\ga K^\al-m\Hi$ is bounded in $L^\infty_{loc}([0,\infty) ;H^1(\Pi_\ep))$ independently of $\ep$.
\end{theorem}
\begin{proof}
We use the decomposition \eqref{decuep} to write
\begin{multline*}
\uep-\ga K^\al-m\Hi=
(1+\alpha \Aep )^{-1} (\K(q_\ep) + m (\Hep-\Hi)) \\
+\ga \bigl[(1+\alpha \Aep )^{-1}\Hep-K^\al\bigr]
+m \bigr[(1+\alpha \Aep )^{-1}\Hi-\Hi\bigr].
\end{multline*}
The first term on the rhs is bounded in $H^1(\Pi_\ep)$ as a consequence of Lemma \ref{H1part}, the second term is bounded in $H^1(\Pi_\ep)$ thanks to Proposition \ref{Hpart} and the $H^1$ bound for the third term follows from Lemma \ref{Hipart}.
\end{proof}

\section{Temporal estimates and passing to the limit}
\label{sect5}

We will now prove the convergence result, which is part b) of Theorem \ref{mainthm}.

We define
\begin{equation}\label{wepdef}
\wep=K_\ep(\qep)+m\H.
\end{equation}

Let $\uept$ be the extension of $\uep$ to the whole of $\R^2$ with zero values for $|x|\leq\ep$. We define in a similar manner $\qept$ and $\wept$. Since $\uep$ and $\wep$ are tangent to the boundary, we infer that $\uept$ and $\wept$ are divergence free in the whole of $\R^2$.

First we note that $K^\al\in H^1_{loc}(\R^2)$. Indeed, this follows from parts (a) and (b) of Lemma \ref{Kal} together with properties (P3) and (P4) of $\mathcal{G}_{\alpha}$. Therefore, using Theorem \ref{uvareW1plocests} together with the vanishing of $\uep$ on the boundary of $\Piep$, we obtain that $\uept-\ga K^\al-m\Hi$ is bounded in $L^\infty_{loc}([0,\infty)  ;H^1(\R^2))$. We infer that there exists some divergence free limit vector field $u$ such that
\begin{equation}\label{uatinf}
u-\ga K^\al-m\Hi\in L^\infty_{loc}([0,\infty)  ;H^1(\R^2)),
\end{equation}
and, passing to subsequences as needed,
\begin{equation}\label{convu}
\uept-u\rightharpoonup 0 \text{ in } L^\infty_{loc}([0,\infty) ;H^1(\R^2))\text{ weak$\ast$ as }\ep\to0.
\end{equation}
In view of the discussion above we have, in particular, that $\uept$ is bounded in $L^\infty_{loc}([0,\infty) ;H^1_{loc}(\R^2))$ and $u$ belongs to $L^\infty_{loc}([0,\infty) ;H^1_{loc}(\R^2))$.

We also find, thanks to Lemma \ref{WvareAsymp}, that $\wept$ is bounded in $L^\infty(\R_+\times\R^2)$ independently of $\ep$, so we can further assume that
\begin{equation}\label{convw}
\wept\rightharpoonup w \quad\text{in }L^\infty(\R_+\times\R^2)\text{ weak$\ast$ as }\ep\to0
\end{equation}
and
\begin{equation}\label{convq}
\qept\rightharpoonup q \quad\text{in }L^\infty(\R_+;L^1(\R^2)\cap L^\infty(\R^2))\text{ weak$\ast$ as }\ep\to0.
\end{equation}

Recall the equation for the potential vorticity
\begin{equation*}
\partial_t \qep+\uep\cdot\nabla\qep=0\quad\text{in }(0,\infty)\times\piep.
\end{equation*}
Since $\uep$ is tangent to the boundary of $\Pi_\ep$ (it even vanishes), the extensions $\uept$ and $\qept$ satisfy the same PDE in all of $\R^2$:
\begin{equation}\label{eqt}
\partial_t \qept+\uept\cdot\nabla\qept=0\quad\text{in }(0,\infty)\times\R^2.
\end{equation}

We observed above that Theorem \ref{uvareW1plocests} implies that $\uept-\ga K^\al-m\Hi$ is bounded in $L^\infty_{loc}([0,\infty)  ;H^1(\R^2))\hookrightarrow L^\infty_{loc}([0,\infty)  ;L^4(\R^2))$. From Lemma \ref{Kal} and the definition of $\Hi$ we observe that $\ga K^\al+m\Hi\in L^4(\R^2)$ and, therefore, $\uept$ is bounded in $L^\infty_{loc}([0,\infty)  ;L^4(\R^2))$. Since $\qept$  is bounded in $L^\infty(\R_+  ;L^4(\R^2))$ we infer that $\uept\qept$ is bounded in the space $L^\infty_{loc}([0,\infty);L^2(\R^2))$, so that   $\uept\cdot\nabla\qept=\dive(\uept\qept)$ is bounded in $L^\infty_{loc}([0,\infty);H^{-1}(\R^2))$. Then $\partial_t \qept=-\uept\cdot\nabla\qept$ is also bounded in $L^\infty_{loc}([0,\infty);H^{-1}(\R^2))$. We infer that the $\qept$ are equicontinuous in time with values in $H^{-1}(\R^2)$. Using the compactness of the embedding $H^{-1}(\R^2)\hookrightarrow H^{-2}_{loc}(\R^2)$ and the Ascoli theorem, we infer that, passing to subsequences if necessary, $\qept\to q$ in $C^0([0,\infty);H^{-2}_{loc}(\R^2)$. Recalling that $\qept$ is bounded in  $L^\infty(\R_+;L^2(\R^2))$ we finally deduce that
\begin{equation}\label{strongconv}
\qept\to q\quad\text{in }C^0([0,\infty); H^{-1}_{loc}(\R^2))\text{ strongly as }\ep\to0.
\end{equation}
The weak convergence  of $\uept$ in $H^1$ obtained in \eqref{convu} then implies that $\uept\qept\to uq$ in the sense of distributions. Thus we also have that $\dive(\uept\qept)\to \dive(uq)$ in the sense of the distributions in $\R^2$. Hence, we can pass to the limit $\ep\to0$ in \eqref{eqt} to obtain that
\begin{equation*}
\partial_t q+u\cdot\nabla q=0\quad\text{in }(0,\infty)\times\R^2.
\end{equation*}
In addition, the strong convergence found in \eqref{strongconv} implies that we also have convergence for the initial data: $\qept(0,\cdot)\to q(0,\cdot)$ as $\ep\to0$ in $H^{-1}_{loc}(\R^2)$. We conclude that we have an initial condition for the equation of $q$:
\begin{equation*}
q(0,x)=q_0(x).
\end{equation*}

Next we will obtain the relationship between $u$ and $q$ expressed in the system of PDE satisfied by $q$. We already know that $u$ is divergence-free, as limit of $\uep$, which are divergence-free. We will proceed to show that $\curl(1-\al\Delta) u=q+\ga\delta$. This will be done in two steps: first we determine the equation for the limit of $\wept$ and then we determine the equation for the limit of $\uept$. The following lemma deals with the first step.
\begin{lemma}\label{lemw}
We have that $w\in L^\infty(\R_+\times\R^2)$, that $\dive w=0$ and that $\curl w=q$.
\end{lemma}
\begin{proof}
We already know that $w\in L^\infty(\R_+\times\R^2)$ and that $\dive w=0$. Let us compute $\curl w$.
Let $\varphi\in C^\infty_c((0,\infty)\times\R^2)$ be a test function and choose some $\mu>0$. Let $\etamu(x)\equiv \eta(x/\mu)$ and $\phimu=\etamu\varphi$, where $\eta$ was introduced on page \pageref{eta}. We assume that $\ep<\mu$. We multiply $\qept$ by $\phimu$ and we integrate by parts, using that $\phimu$ is compactly supported in $\Pi_\ep$ and that $\qep=\curl\wep$:
\begin{align*}
\int_0^\infty\int_{\R^2}\qept\phimu
&=\int_0^\infty\int_{\Pi_\ep}\qep\phimu\\
&=\int_0^\infty\int_{\Pi_\ep}\curl\wep\phimu\\
&=-\int_0^\infty\int_{\Pi_\ep}\wep\cdot\nabla^\perp\phimu\\
&=-\int_0^\infty\int_{\R^2}\wept\cdot\nabla^\perp\phimu.
\end{align*}
We send $\ep\to0$ and use the weak convergences found in \eqref{convw} and in \eqref{convq} to obtain that
\begin{equation*}
\int_0^\infty\int_{\R^2}q\phimu=-\int_0^\infty\int_{\R^2}w\cdot\nabla^\perp\phimu.
\end{equation*}
One can easily check that $\phimu\to\varphi$ weakly in $H^1$ as $\mu\to0$. Recalling that $w$ is bounded in space and time and, hence, it belongs to $L^2_{loc}$, and observing that $\phimu$ is supported in a compact set independent of $\mu$, we can pass to the limit $\mu\to0$ above to obtain that
\begin{equation*}
\int_0^\infty\int_{\R^2}q\varphi=-\int_0^\infty\int_{\R^2}w\cdot\nabla^\perp\varphi.
\end{equation*}
This means that $\curl w=q$ in the sense of the distributions. This concludes the proof of the lemma.
\end{proof}

The second step consists in computing $u-\al\Delta u$ in terms of $w$.
\begin{lemma}\label{propu}
We have that $u-\al\Delta u=w+\ga H$.
\end{lemma}
\begin{proof}
Let $\Psi\in C^\infty_{c,\sigma}((0,\infty)\times\R^2;\R^2)$ be a divergence-free test vector field and choose some $\mu>0$. We assume $\ep<\mu$.

Since $\Psi$ is divergence-free, there exists $\Phi\in C^\infty((0,\infty)\times\R^2;\R)$, compactly supported in time, such that $\Psi=\nabla^\perp\Phi$. We can assume, without loss of generality, that $\Phi(t,0)=0$ for all $t$.

As before, let $\etamu(x)=\eta(x/\mu)$ and $\Psi_\mu=\nabla^\perp(\etamu\Phi)$. Then $\Psi_\mu$ is divergence free and compactly supported in $\Pi_\ep$ so $(1+\al\Aep)\Psi_\mu=(1-\al\Delta)\Psi_\mu$. Recall that $(1+\al\Aep)\uep=\wep+\ga H$, see \eqref{mBS} and the definition of $\wep$ given in relation \eqref{wepdef}. We write
\begin{align*}
\int_0^\infty\int_{\R^2}\uept\cdot (\Psi_\mu-\al\Delta\Psi_\mu)
&=\int_0^\infty\int_{\Pi_\ep}\uep\cdot (1+\al\Aep)\Psi_\mu\\
&=\int_0^\infty\int_{\Pi_\ep}(1+\al\Aep)\uep\cdot \Psi_\mu\\
&=\int_0^\infty\int_{\Pi_\ep}\wep\cdot \Psi_\mu+\ga \int_0^\infty\int_{\Pi_\ep}\H\cdot \Psi_\mu\\
&=\int_0^\infty\int_{\R^2}\wept\cdot \Psi_\mu+\ga \int_0^\infty\int_{\R^2}\H\cdot \Psi_\mu.
\end{align*}
We now let $\ep\to0$ and use \eqref{convu} and \eqref{convw} to pass to the limit. We obtain
\begin{align*}
\int_0^\infty\int_{\R^2}u\cdot (\Psi_\mu-\al\Delta\Psi_\mu)
=\int_0^\infty\int_{\R^2}w\cdot \Psi_\mu+\ga \int_0^\infty\int_{\R^2}\H\cdot \Psi_\mu.
\end{align*}
We rewrite the last term above in the following form:
\begin{align*}
\int_0^\infty\int_{\R^2}\H\cdot \Psi_\mu
&=\int_0^\infty\int_{\R^2}\H\cdot \nabla^\perp(\etamu\Phi)\\
&=\int_0^\infty\int_{\R^2}\H\cdot \nabla^\perp((\etamu-1)\Phi)+\int_0^\infty\int_{\R^2}\H\cdot \nabla^\perp\Phi\\
&=\int_0^\infty\int_{\R^2}\H\cdot \nabla^\perp((\etamu-1)\Phi)+\int_0^\infty\int_{\R^2}\H\cdot \Psi.
\end{align*}
We now use that $\curl\H=\delta$, we recall that $\etamu-1$ is $C^\infty$ and compactly supported, and we write the following sequence of equalities in the sense of the distributions $\mathscr{D}'(\R^2)$:
\begin{multline*}
\int_0^\infty\int_{\R^2}\H\cdot \nabla^\perp((\etamu-1)\Phi)
=\int_0^\infty\langle \H,\nabla^\perp((\etamu-1)\Phi)\rangle
=-\int_0^\infty\langle \curl\H,(\etamu-1)\Phi\rangle\\
=-\int_0^\infty\langle \delta,(\etamu-1)\Phi\rangle
=\int_0^\infty\Phi(t,0)\,dt=0,
\end{multline*}
where we also used that $\etamu(0)=0$.

We infer that
\begin{multline*}
\int_0^\infty\int_{\R^2}(u\cdot \Psi_\mu+\al\nabla u\cdot\nabla\Psi_\mu)
=\int_0^\infty\int_{\R^2}u\cdot (\Psi_\mu-\al\Delta\Psi_\mu)\\
=\int_0^\infty\int_{\R^2}w\cdot \Psi_\mu+\ga \int_0^\infty\int_{\R^2}\H\cdot \Psi.
\end{multline*}
Since $\Phi(t,0)=0$ one can easily check that $\Psi_\mu\to\Psi$ in $L^\infty(\R_+;H^1(\R^2))$ weak$\ast$ and, moreover, the support of $\Psi_\mu$ is included in a compact set uniformly with respect to $\mu$. Since $u\in L^\infty_{loc}([0,\infty_;H^1_{loc}(\R^2))$ one can pass to the limit $\mu\to0$ above to obtain
\begin{equation*}
\int_0^\infty\int_{\R^2}(u\cdot \Psi+\al\nabla u\cdot\nabla\Psi)
=\int_0^\infty\int_{\R^2}w\cdot \Psi+\ga \int_0^\infty\int_{\R^2}\H\cdot \Psi.
\end{equation*}
This can be written in the following form in the sense of the distributions:
\begin{equation*}
\langle u-\al\Delta u-w-\ga H,\Psi\rangle=0
\end{equation*}
for all divergence-free test vector fields $\Psi\in C^\infty_{c,\sigma}((0,\infty)\times\R^2;\R^2)$. Since the vector field $ u-\al\Delta u-w-\ga H$ is divergence free, we deduce from the relation above that it must vanish. This completes the proof of the proposition.
\end{proof}

Recall that $\curl H=\delta$ in $\R^2$. Then, by virtue of Lemmas \ref{lemw} and \ref{propu}, it follows that
\begin{equation} \label{needsanumber}
\curl(u-\al\Delta u)=\curl(w+\ga \H)=\curl w+\ga\delta=q+\ga\delta.
\end{equation}

We have shown the convergence of a subsequence of $\qep$ towards a solution of \eqref{limitsystem}. In the next section we will show that the solutions of \eqref{limitsystem} are unique, which, in turn, implies that the full sequence $\qep$ converges to $q$, without the need to pass to a subsequence.

To conclude the proof of part b) of Theorem \ref{mainthm} it remains to show that $u \in L^\infty_{loc}(\R_+;L^p(\R^2))$ for any $p>2$. This follows immediately from \eqref{uatinf} since it is easy to see, using Lemma \ref{Kal} part b) and the definition of $\Hi$, that both $K^\al$ and $\Hi \in L^p(\R^2)$, for all $p>2$, and since $H^1(\R^2) \subset L^p(\R^2)$ for all $p\geq 1$.

\begin{remark}\label{remark}
Recall that the kernel of the solution operator in the full plane for $\curl(1-\al\Delta)$, on divergence-free vector fields vanishing at infinity, is $K^\al=\mathcal{G}_\alpha \ast H$. This solution operator can be easily extended to solenoidal vector fields in $L^p(\R^2)$.  Therefore, using \eqref{needsanumber} and that $u \in L^\infty_{loc}(\R_+;L^p(\R^2))$, $p>2$, we infer that $u$ can be expressed as:
\begin{equation*}
u=K^\al\ast(q+\ga\delta)=K^\al\ast q+\ga K^\al.
\end{equation*}
Moreover, if we denote by $\check{v}=H\ast q$ the velocity field associated to $q$ in $\R^2$ then one can check that
\begin{equation*}
\curl(\partial_t \check{v}+u\cdot\nabla \check{v}+\sum_j \check{v}_j\nabla u_j)=\partial_t \curl \check{v}+u\cdot\nabla \curl \check{v}=\partial_t q+u\cdot\nabla q=0.
\end{equation*}
So the velocity formulation of \eqref{limitsystem} can be written in the form
\begin{align*}
\partial_t \check{v}+u\cdot\nabla \check{v}+\sum_j \check{v}_j\nabla u_j+\nabla p&=0\\
\dive u&=0\\
u-\al\Delta u&=\check{v}+\ga H.
\end{align*}
\end{remark}

\section{Uniqueness for the limit system}
\label{sect6}

The global existence of  solutions of \eqref{limitsystem} follows from the convergence result established in the previous section. Here we will prove uniqueness of solutions of \eqref{limitsystem}, thereby completing the proof of Theorem \ref{mainthm}.

Let us observe that the $\al$--Euler equations in $\R^2$, with an initial vorticity given by a bounded measure such as $q_0+\ga\delta$, have a global unique solution, see \cite{oliver_vortex_2001}. But, as noted in the introduction, even though the limit system \eqref{limitsystem} is very similar to the $\alpha$--Euler system, it is not the same. In addition, \cite{oliver_vortex_2001} proves uniqueness of Lagrangian solutions by working on the trajectories of the velocity field. Even if we could adapt the proof of uniqueness for $\al$--Euler to \eqref{limitsystem}, we would still have to make the connection between Lagrangian solutions and the Eulerian solutions considered here. Instead, we will give below a more classical proof of uniqueness, based on energy estimates.

Let $q,q'\in L^\infty(\R_+;L^1(\R^2)\cap L^\infty(\R^2))$ be two solutions of the limit system \eqref{limitsystem} with the same initial data $q(0,x)=q'(0,x)=q_0(x)$. Then $u=K^\al\ast q+\ga K^\al$ and $u'=K^\al\ast q'+\ga K^\al$, see Remark \ref{remark}.

Let $\qb=q-q'$, $\ub=u-u'$ and $\vb=H\ast\qb=\nabla^\perp\Delta^{-1}\qb$.
Clearly $\int_{\R^2}q\,dx=\int_{\R^2}q_0\,dx=\int_{\R^2}q'\,dx$ so $\int_{\R^2}\qb\,dx=0$. In addition, $q$ and $q'$ are obviously compactly supported in space. We infer that $\vb$ decays like $O(1/|x|^2)$ at infinity, so that it belongs to $L^2$.

We have the following PDE for $\qb$:
\begin{equation*}
\partial_t\qb+u\cdot\nabla q-u'\cdot\nabla q'=0.
\end{equation*}
We multiply by $\Delta^{-1}\qb$ and integrate in $[0,T]\times\R^2$. We follow the same argument as in Section \ref{sect3}, when we estimated $\frac{\d}{\d t}\|\vb^{n+1}-\vb^n\|_{L^2(\Piep)}^2$. Redoing the same estimates as those found on pages \pageref{page1}--\pageref{page2} we find
\begin{align*}
\frac12\|\vb(T)\|_{L^2}^2
&=-\int_0^T\int_{\R^2}(\partial_2u_1+\partial_1u_2)\vb_1\vb_2
-\frac12 \int_0^T\int_{\R^2}(-\partial_1u_1+\partial_2u_2)(\vb_2^2-\vb_1^2)\\
&\hskip 6cm+\int_0^T\int_{\R^2}\ub\cdot\vb^\perp \,q'\\
&\leq C\int_0^T (\|\partial_2u_1+\partial_1u_2\|_{L^\infty}+\|\partial_1u_1-\partial_2u_2\|_{L^\infty})\|\vb\|_{L^2}^2\\
&\hskip 6cm+C\int_0^T\|\ub\|_{L^2}\|\vb\|_{L^2}\|q'\|_{L^\infty}\\
&\leq C\int_0^T (1+\|\partial_2u_1+\partial_1u_2\|_{L^\infty}+\|\partial_1u_1-\partial_2u_2\|_{L^\infty})\|\vb\|_{L^2}^2,
\end{align*}
where we used the relation $\ub-\al\Delta\ub=\vb$ to bound $\|\ub\|_{L^2}\leq \|\vb\|_{L^2}$ and, also, that $q'$ is bounded in $L^\infty$.

We now use \eqref{newest} to obtain
\begin{align*}
\|\partial_1u_1-\partial_2u_2\|_{L^\infty}
&\leq \|(\partial_1K^\al_1-\partial_2K^\al_2)\ast q\|_{L^\infty} +|\ga|\|\partial_1K^\al_1-\partial_2K^\al_2\|_{L^\infty}\\
&\leq \|(\partial_1K^\al_1-\partial_2K^\al_2)\|_{L^\infty}(\|q\|_{L^1} +|\ga|)\\
&\leq C,
\end{align*}
where $C$ is uniform in time. A similar estimate holds true for
$\|\partial_2u_1+\partial_1u_2\|_{L^\infty}$. We deduce that
\begin{equation*}
\|\vb(T)\|_{L^2}^2\leq C\int_0^T \|\vb\|_{L^2}^2.
\end{equation*}
The Gronwall inequality then implies that $\vb=0$, so that $q=q'$. This completes the proof of Theorem \ref{mainthm}.

\vspace{.5cm}

\scriptsize{
\textbf{Acknowledgments.}   M. C. Lopes Filho and H.J. Nussenzveig Lopes thank the PICS \#288801 (PICS08111)  of the CNRS, for their financial support for the scientific visits which led to this work. M.C. Lopes Filho acknowledges the support of CNPq through grant \# 310441/2018-8 and of FAPERJ through grant \# E-26/202.999/2017 . H.J. Nussenzveig Lopes thanks the support of CNPq through grant \# 309648/2018-1  and of FAPERJ through grant \# E-26/202.897/2018.
M. C. Lopes Filho and H. J. Nussenzveig Lopes thank both the Universit\'e de Lyon 1 and the Universit\'e de Saint-\'Etienne, where part of this work was done. In addition, M. C. Lopes Filho and H. J. Nussenzveig Lopes would like to thank the Isaac Newton Institute for Mathematical Sciences for support and hospitality during the programme ``Mathematical aspects of turbulence: where do we stand?". This work was supported in part by:

\vspace{.1cm}

\noindent EPSRC Grant Number EP/R014604/1.
}


 \bigskip

\begin{description}
\item[Adriana Valentina Busuioc] Université de Lyon, Université de Saint-Etienne  --
CNRS UMR 5208 Institut Camille Jordan --
Faculté des Sciences --
23 rue Docteur Paul Michelon --
42023 Saint-Etienne Cedex 2, France.\\
Email: \texttt{valentina.busuioc@univ-st-etienne.fr}\\
Web page: \texttt{https://perso.univ-st-etienne.fr/busuvale/}
\item[Dragoş Iftimie] Université de Lyon, CNRS, Université Lyon 1, Institut Camille Jordan, 43 bd. du 11 novembre, Villeurbanne Cedex F-69622, France.\\
Email: \texttt{iftimie@math.univ-lyon1.fr}\\
Web page: \texttt{http://math.univ-lyon1.fr/\~{}iftimie}
\item[Milton C. Lopes Filho] Instituto de Matem\'atica,
Universidade Federal do Rio de Janeiro,
Cidade Universit\'aria -- Ilha do Fund\~ao,
Caixa Postal 68530,
21941-909 Rio de Janeiro, RJ -- Brasil. \\
Email: \texttt{mlopes@im.ufrj.br} \\
Web page: \texttt{http://www2.im.ufrj.br/mlopes}
\item[Helena J. Nussenzveig Lopes] Instituto de Matem\'atica,
Universidade Federal do Rio de Janeiro,
Cidade Universit\'aria -- Ilha do Fund\~ao,
Caixa Postal 68530,
21941-909 Rio de Janeiro, RJ -- Brasil. \\
Email: \texttt{hlopes@im.ufrj.br}\\
Web page: \texttt{http://www2.im.ufrj.br/hlopes}
\end{description}

\end{document}